\documentclass[12pt]{amsart}
\usepackage{graphicx}
\usepackage{amssymb}
\usepackage{enumerate}
\usepackage{amstext}
\usepackage[colorlinks]{hyperref}
\usepackage[cp1250]{inputenc}

\newcommand{\e}{\varepsilon}
\renewcommand{\rho}{\varrho}
\newcommand{\ph}{\varphi}
\newcommand{\N}{\mathbb N}
\newcommand{\R}{\mathbb R}
\newcommand{\K}{\mathcal{K}}
\newcommand{\F}{\mathcal{F}}
\newcommand{\G}{\mathcal{G}}
\newcommand{\X}{X}
\newcommand{\Lip}{\mathrm{Lip}}
\newcommand{\dist}{\operatorname{dist}}
\newcommand{\diam}{\operatorname{diam}}
\newcommand{\h}{\operatorname{ht}}
\newcommand{\rk}{\operatorname{rk}}
\newcommand{\LIM}{\operatorname{LIM}}
\newtheorem {theorem}{Theorem}
\newtheorem {lemma}{Lemma}
\newtheorem {claim}{Claim}
\newtheorem {corollary}{Corollary}
\theoremstyle{definition}
\newtheorem {definition}{Definition}

\begin{document}

\title{Topological classification of scattered IFS-attractors}
\author{Magdalena Nowak}
\address{Faculty of Mathematics and Computer Science\\ Jagiellonian University\\
ul.\L o\-jasiewicza 6, 30-348 Krak\'ow, Poland \\{\rm and}\\Institute of Mathematics\\Jan Kochanowski University\\ul. \'Swi\c{e}tokrzyska 15, 25-406 Kielce, Poland}
\email{magdalena.nowak805@gmail.com}

\date{\today}
%\thanks{The author was supported by the ESF Human Capital Operational Programme grant 6/1/8.2.1./POKL/ 2009.}

\begin{abstract}
We study countable compact spaces as potential attractors of iterated function systems.
We give an example of a convergent sequence in the real line which is not an IFS-attractor and for each countable ordinal $\delta$ we show that a countable compact space of height $\delta+1$ can be embedded in the real line so that it becomes the attractor of an IFS.
On the other hand, we show that a scattered compact metric space of limit height is never an IFS-attractor.
\end{abstract}


\maketitle

\tableofcontents

\section{Introduction}

A compact metric space $X$ is called an {\em IFS-attractor} if
$$X=\bigcup_{i=1}^nf_i(X)$$ for~some contractions $f_1,\dots,f_n:X\to X$.
The family $\{f_1,\dots,f_n\}$ is~called an {\em iterated function system} (briefly, an IFS), see \cite{B}.
When $X$ is a subset of some Euclidean space, usually it is assumed that the contractions are defined on the entire space.
In 1981 Hutchinson showed that for the metric space $\R^n$, every IFS consisting of contractions has a unique attractor. In fact this theorem holds also for weak contractions \cite{Hata}.
Let us say that $X$ is a \emph{topological IFS-attractor} if it is homeomorphic to the attractor of some iterated function system or, in other words, there exists a compatible metric on $X$ such that $X$ becomes an IFS-attractor.
Given a metrizable compact space, it is natural to ask when it is a topological IFS-attractor.
A criterion for connected spaces has already been noted by Hata~\cite{Hata}: A connected IFS-attractor must be locally connected.
Recently, Banakh and the author \cite{BanakhNowak} gave an example of a connected and locally connected compact subset of the plane that is not a topological IFS-attractor.

%It is also natural to ask whether some infinite compact metrizable space is an IFS-attractor with respect to any compatible metric.
On the other hand, a result of Kwieci\'nski~\cite{kwiecinski}, later generalized by Sanders \cite{Sanders}, shows the existence of a curve in the plane that is not an IFS-attractor.
In other words, the unit interval (which is obviously an IFS-attractor) has a compatible metric (taken from the plane) such that it fails being an IFS-attractor.
In this direction, Kulczycki and the author~\cite{KulczyckiNowak} gave a general condition on a connected compact space which implies that it has a compatible metric making it a non-IFS-attractor.
Finally, \cite{Rams} showed that the Cantor set has a metric such that it fails to be the attractor of even a countable system of contractions.

Motivated by these results, we study topological properties of scattered IFS-attractors.
It is easy to see that each finite set is an IFS-attractor in every metric space. We present an example of a convergent sequence of real numbers (a~countable compact set in $\R$), which is not an IFS-attractor.
We further investigate more complicated scattered compact spaces and classify them with respect to the property of being a topological IFS-attractor.
Namely, we show that every countable compact metric space of successor Cantor-Bendixson height with a single point of the maximal rank can be embedded topologically in the real line so that it becomes the attractor of an IFS consisting of two contractions whose Lipschitz constants are as small as we wish.
On the other hand, we show that if a countable compact metric space is a topological IFS-attractor, then its Cantor-Bendixson height cannot be a limit ordinal.

Combining our results, we get an example of a countable compact metric space $\K$ (namely, a space of height $\omega+1$) which is an IFS-attractor, however some clopen subset of $\K$ is not an IFS-attractor, even after changing its metric to an equivalent one.

\section{Preliminaries}

Throughout the paper we will use the following standard notation: $d$ will stand for a metric or for the usual distance on the real line; the distance between sets $A,B\subset\R$ will be denoted by $\dist(A,B)= \inf\{d(a,b) \colon a\in{A},~~b\in{B}\}$ and $\diam{A}$ will denote the diameter of~the~set $A\subset\R$: $\diam{A}=\sup_{x,y\in{A}}\{d(x,y)\}$. Let $|A|$ will stand for the number of~elements in the set $A$ and $A+x$ be the set $\{a+x \colon a\in A \}$. Finally $B(x,r)$ will denote the~open ball of radius $r > 0$ centered at~the~point~$x$. 
Given a metric space $(X,d)$, a map $f\colon X \rightarrow X$ is called a {\em contraction} if there exists a constant $\alpha\in(0, 1)$ such that for each $x,y\in X$
$$d(f(x),f(y))\leq\alpha\cdot d(x, y).$$
A map $f\colon X \rightarrow X$ is called a {\em weak contraction} if for each $x,y\in X$, $x\neq y$
$$d(f(x),f(y))< d(x, y).$$
It is well known that every weak contraction on a compact metric space has a unique fixed point.

We recall some basic notions related to scattered spaces.
A topological space X is called {\em scattered} iff every non-empty subspace Y has an
isolated point in Y. It~is well known that a~compact metric space is scattered iff it is countable. Moreover every compact scattered space is zero-dimensional (has a base consisting of clopen sets). 

For a scattered space $X$ let
$$X'=\{x\in X \colon x\text{ is an accumulation point of } X\}$$
be the Cantor-Bendixson derivative of $X$. Inductively define:
\begin{itemize}
\item $X^{(\alpha+1)}=(X^{(\alpha)})'$
\item $X^{(\alpha)}=\bigcap_{\beta<\alpha}X^{(\beta)}$ for a limit ordinal $\alpha$.
\end{itemize}
In general, the set $X^{(\alpha)}\setminus X^{(\alpha+1)}$ is called the $\alpha$th \emph{Cantor-Bendixson level} of $X$.
For an element $x$ of a scattered space $X$, its \emph{Cantor-Bendixson rank} $\rk(x)$ is the unique ordinal $\alpha$ such that $x \in X^{(\alpha)} \setminus X^{(\alpha+1)}$.
The \emph{height} of a scattered space $X$ is $$\h(X) = \min\{\alpha \colon X^{(\alpha)} \text{ is discrete}\}.$$

These are topological invariants of scattered spaces and their elements. By the~definitions and~transfinite induction it~is easy to~prove that for every compact scattered spaces $U$ and $V$
\begin{itemize}
\item if $U\subset V$ then $\h(U)\leq \h(V)$ 
\item $\h(U\cup V) = \max(\h(U),\h(V))$
\item $\h(f(U))\leq \h(U)$ for every continuous function $f$
\item $\h(U)\geq \rk(x)$ for every open neighborhood $U$ of $x$
\end{itemize}   
 The classical Mazurkiewicz-Sierpi\'nski theorem \cite{MS} claims that every countable compact scattered space $X$ is homeomorphic to the space $\omega^{\beta}\cdot n + 1$ with the order topology, where $\beta=\h(X)$ and $n=|X^{(\beta)}|$ is finite. We~shall consider scattered compact spaces of that form.

We finally note two simple properties of disjoint unions of IFS-attractors, which will be needed later.

\begin{lemma}\label{Lsumadwa}
Suppose $X = \bigcup_{i < n} X_i$ is a metric space, where each $X_i$ is compact and isometric to $X_0$ and $\dist(X_i, X_j) > \diam(X_0)$ for every $i<j<n$.
If $X$ is an IFS-attractor (consisting of weak contractions) then so is $X_0$.
\end{lemma}

\begin{proof}
Let $\{f_i\}_{i=1}^{k}$ be an IFS such that $X = \bigcup_{i=1}^k f_i(X)$.
Note that if $f$ is a weak contraction and $f(X_i) \cap X_0 \ne \emptyset$ then $f(X_i) \subset X_0$, because $\dist(X_0, X_i) > \diam(X_0) = \diam(X_i)$.
For each $i$ let $h_i$ be an isometry from $X_0$ onto $X_i$.
Denote by $S$ the set of all pairs $(i,j)$ such that $f_i(X_j) \subset X_0$.
By the remark above, $X_0 = \bigcup_{(i,j) \in S}f_i(X_j)$.
Thus, $X_0$ is the attractor of an IFS consisting of (weak) contractions of the form $f_i \circ h_j$ where $(i,j) \in S$.
\end{proof}

\begin{lemma}\label{LemYUnion}
Assume $X = A \cup B$ is a compact metric space, where $A,B$ are clopen and disjoint IFS-attractors.
Then $X$ is an IFS-attractor.
\end{lemma}

\begin{proof}
Given an IFS $\F$, given $k \in \N$, denote by $\F^k$ the collection of all compositions $f_1 \circ f_2 \circ \dots \circ f_k$, where $f_1, \dots, f_k \in \F$ (possibly with repetitions).
Then $\F^k$ is another IFS with the same attractor.
Moreover, if $r = \max_{f \in \F} \Lip(f)<1$ then $r^k \geq \max_{g \in \F^k} \Lip(g)$.

We may assume that both sets $A$, $B$ are nonempty and that $1 = \diam(X)$.
Let $\e = \dist(A,B)$.
In view of the remark above, we may find two iterated function systems $\F$ and $\G$ on $A$ and $B$ respectively, such that $A$ and $B$ are their attractors, and the maximum of all Lipschitz constants of the contractions in $\F$ and $\G$ is $< \frac 12 \e$.
In particular, $\diam(h(A)) < \frac 12 \e$ whenever $h \in \F$ and $\diam(h(B)) < \frac 12 \e$ whenever $h \in \G$.

Extend each $f \in \F$ to a map $f' \colon X \to X$ by letting $f'(B) = \{ p_f \}$, where $p_f$ is any fixed element of $f(A)$.
Observe that the Lipschitz constant of $f'$ is $\leq \frac 12$, because given $x \in A$, $y \in B$, we have
$$d(f'(x), f'(y)) \leq \diam(f(A)) < \frac 12 \e = \frac 12 \dist(A,B) \leq \frac 12 d(x,y).$$
Similarly, extend each $g \in \G$ to a map $g'$ so that $g'(A) = \{ p_g \}$, where $p_g \in g(B)$.
Again, $g'$ has Lipschitz constant $\leq \frac 12$.

Finally, $\{f'\}_{f\in \F} \cup \{g'\}_{g\in \G}$ is an IFS whose attractor is $X$.
\end{proof}

It is a natural question whether the converse to Lemma~\ref{LemYUnion} holds.
As we shall see later, this is not the case.

\section{Convergent sequences}

In this section we construct a convergent sequence in the real line, which is not an IFS-attractor.

When we consider a sequence $\{x_n\}_{n\in\N}$ as a possible attractor of an iterated function system in $\R$, we identify that sequence with the closure $\overline{\{x_n \colon n\in\N\}}$. We say that a sequence $\{x_n\}_{n\in\N}$ is an IFS-attractor if so is $\overline{\{x_n \colon n\in\N\}}$.
\newline

Every geometric convergent sequence is~an~IFS-attractor. For example the set $\{0\}\cup\{\frac{1}{2^n}\}_{n\in\N}$ is an attractor of~IFS $\{f_1(x)=\frac{x}{2} , f_2(x)=1\}$. We will give an~example of a convergent sequence which is not the attractor of any IFS in $\R$.

\begin{theorem}\label{Tw}
There exists a convergent sequence $\K\subset\R$ which is not the attractor of any iterated function system in $\R$ consisting of weak contractions.
\end{theorem}

The construction of $\K$ is inspired by the example of a locally connected continuum which is not the attractor of any IFS on $\R^2$, constructed by Kwieci{\'n}ski~\cite{kwiecinski}.

We construct the sequence $\K$ as follows. The main building block is the set $F(a,k)$, for $a>0$ and $k\in\N$, defined by 
$$F(a,k) = \Big\{\frac{ia}k \colon i=0,...,k-1\Big\}~.$$
Note that for every distinct $x,y\in{F(a,k)}$ we have that $d(x,y)\geq\frac{a}{k}> 0$, therefore if $d(x,y) < \frac{a}{k}$, then $x=y$.
\newline

Now, let $a_n = \frac{1}{3\cdot 2^{n-1}}$ and $k_n= n(k_{n-1}+\dots+k_1)$ where $k_1=1$. Then
$$F_n = F(a_n,k_n) + \frac{1}{2^{n-1}}~.$$

\begin{figure}[h]
	\includegraphics[scale =0.8]{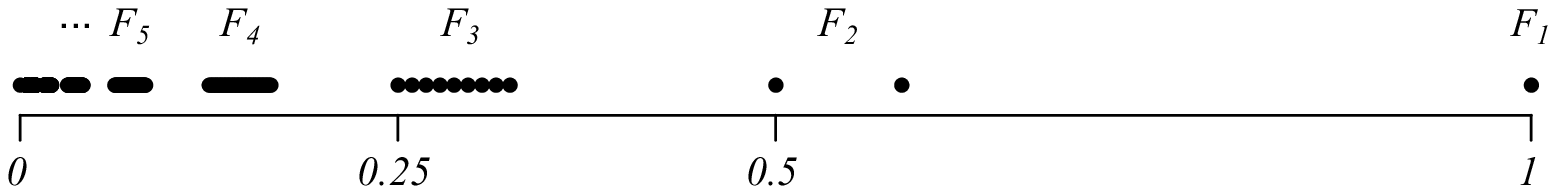}
	\caption{The sequence $\K$}
\end{figure}

The set $\K$ is defined to be the union $$\K = \{0\}\cup\bigcup_{n=1}^\infty F_n~.$$

It is clear that $\K$ consists of a decreasing sequence and its limit point.
Note that:
\begin{enumerate}
\item\label{1} the sequence \{$\frac{a_n}{k_n}\}_{n\in\N^+}$ is decreasing
\item\label{2} the sequence $\{\dist(F_n,F_{n+1})\}_{n\in\N^+}$ is decreasing
\item\label{3} for all $n\in\N^+$ we have $$\diam{F_n}\leq{a_n}< \dist(F_n,F_{n+1}).$$
%In particular properties (\ref{2}) and (\ref{3}) imply that 
%\item\label{4} for $x\in{F_n}$ we have $\K\cap B(x,\dist(F_n,F_{n+1}))=F_n$.
\end{enumerate}

The idea behind the construction of $\K$ is that weak contractions on that set behave in a specific way. In particular we have the following

\begin{lemma}\label{kontrakcje}
For a weak contraction $f\colon\K\rightarrow\K$ either $$f(F_n)\subset\K\setminus(F_1\cup\dots\cup F_n) \text{ for all }n\in\N^+$$
or else the set $f(\K)$ is finite.
\end{lemma}

\begin{proof}
Let $f$ be a weak contraction on $\K$ satisfying $f(0)\neq 0$. This means that $f(0)$ is an isolated point of $\K$. The function $f$ is continuous, so there exists an~open neighborhood $U$ of 0, such that $f(U)=\{f(0)\}$. Thus, the~set $f(\K)=f(U)\cup f(\K\setminus U)$ is finite.
\newline

If $f(0)=0$, for each $n\in\N^+$ there exists $x\in F_n$ such that $d(0,x)=d(0, F_n)$. Then $d(0,f(x))<d(0,x)=d(0, F_n)$ which implies $f(x)\in\K\setminus(F_n\cup\dots\cup F_1)$ and by (\ref{3}) we have $$\diam(f(F_n))<\diam(F_n)<\dist(F_n,F_{n+1})=\dist(F_n,\bigcup_{i=n+1}^\infty F_i).$$
This implies that $f(F_n)\cap(F_1\cup\dots\cup F_n)=\emptyset$.
\end{proof}

\begin{proof}[Proof of Theorem ~\ref{Tw}]
Suppose that $\K$ is the attractor of an iterated function system  $\F=~\{f_1,f_2,\dots,f_r\}$ consisting of weak contractions in $\R$. That is, $\K=~\bigcup_{i=1}^r{f_i(\K)}$. By Lemma~\ref{kontrakcje}, we know that there are two kinds of weak contractions $f$ on $\K$:
\begin{enumerate}[(i)]
\item\label{1rodz} $f(\K)$ is finite
\item\label{2rodz} for all $n\in\N^+$ it holds that $f(F_n)\subset \K\setminus(F_{1}\cup\dots\cup{F_n})$.
\end{enumerate}

Now we can write the set $\K$ as the union $\K=\bigcup_{i=1}^m{f_i(\K)}\cup{S}$ where $m\leq{r}$, the~functions $f_i$ for $i=1,\dots,m$ satisfy (\ref{2rodz}) and the set $S=\bigcup_{i=m+1}^rf_i(\K)$ is finite.
This implies that $$F_n\subset\bigcup_{i=1}^m f_i(F_{n-1}\cup\dots\cup{F_1})\cup S~.$$ Indeed, if $x\in F_n$ then $x=f(y)$ for some $f\in\F$ and $y\in\K$. If $f$ is of type (\ref{1rodz}) then $x\in S$. Otherwise $y\in F_{n-1}\cup\cdots\cup F_1$, because of (\ref{2rodz}).

Since $S$ is finite, for $n$ big~enough we have that $F_n\subset\bigcup_{i=1}^m f_i(F_{n-1}\cup\dots\cup{F_1})$ so
$$k_n=|F_n| \leq |\bigcup_{i=1}^m f_i(F_{n-1}\cup\dots\cup{F_1})| \leq m(k_{n-1}+\dots+k_1)~.$$
But $k_n= n(k_{n-1}+\dots+k_1)$ so for $n>m$ we get a contradiction.   
\end{proof}

In fact, every compact scattered space can be embedded topologically in the real line so that its image is not the attractor of any IFS consisting of weak contractions. We prove the result below, using the same idea as for the convergent sequence.

\begin{theorem}\label{Tw_noIFSattractor}
A compact scattered metric space with successor height can be embedded topologically in the real line so that it is not the attractor of any iterated function system consisting of weak contractions.
\end{theorem}

\begin{proof}
First, we use the idea of the proof of Theorem \ref{Tw} for the space homeomorphic to $\omega^\delta+1$, where $\delta=\alpha+1$ is a fixed successor ordinal.

Let us consider such space written as $\X=\{0\}\cup\bigcup_{n=1}^\infty X_n$, where each space $X_n$ is homeomorphic to $\omega^\alpha+1$ and $X_n \cap X_m = \emptyset$ whenever $n \ne m$.
We can topologically embed the space $\X$ into the real line such that:
\begin{itemize}
\item the spaces $X_k$ are gathered in blocks $\{F_n\}_{n=1}^\infty$ which accumulate to~0;
\item each block $F_n$ contains $k_n$ spaces of the form $X_k$;
\item for every $n\geq 1$ it holds that
\begin{equation}
 \diam(F_n)\leq\dist(F_n,F_{n+1}).
\tag{*}\label{gwiazdka}
\end{equation}
\end{itemize}
In other words, we take the convergent sequence constructed in the proof of Theorem~\ref{Tw} and replace each point by a copy of $\omega^\alpha + 1$, taking care that its diameter should be small enough, so that (\ref{gwiazdka}) holds.

As in the proof of Theorem \ref{Tw}, we may show that there are two kinds of weak contractions $f$ on $\X$:
\begin{enumerate}[(i)]
\item\label{1kontr} either $f(\X)$ covers only finitely many sets $X_n$, or
\item\label{2kontr} for all $n\geq 1$ we have that $f(F_n)\subset \X\setminus(F_{n}\cup\dots\cup{F_1})$.
\end{enumerate}
To show this dichotomy we have to use (\ref{gwiazdka}) and the fact that it is impossible to cover the space $\X$ using finitely many spaces of height $<\delta$.

Now we can omit contractions of the first type, as in the proof of Theorem \ref{Tw}, and we observe that if $\X$ is an IFS-attractor, then for $n$ big enough we have that
$$F_n\subset\bigcup_{i=1}^m f_i(F_{n-1}\cup\dots\cup{F_1})$$ 
where $f_i$ for $i=1,\dots,m$ satisfy (\ref{2kontr}). Then 
$$k_n=|F_n^{(\alpha)}| \leq\Big|\Big(\bigcup_{i=1}^m f_i(F_{n-1}\cup\dots\cup{F_1})\Big)^{(\alpha)}\Big| \leq m(k_{n-1}+\dots+k_1).$$
Now, taking $n>m$ we get a contradiction by the definition of $k_n$.   
\newline

We have already shown that every space $\omega^{\delta}+1$ of successor height can be embedded topologically in the real line so that it is not the attractor of any IFS. To show that each $\omega^{\delta}\cdot n +1$ has the same property, we place on the real line $n$ isometric copies $X_1, \dots, X_n$ of the space constructed before (homeomorphic to $\omega^{\delta}+1$) so that
$$\diam(X_k)=\diam(X_{k+1})<\dist(X_k,X_{k+1})$$
for every $k=1,\dots,n-1$.
By Lemma~\ref{Lsumadwa} we conclude that if $X_1$ is not the attractor of any IFS then neither is $X = X_1\cup \dots \cup X_n$.
\end{proof}

We have proved Theorem \ref{Tw_noIFSattractor} only for compact scattered spaces of successor height. It turns out that spaces of limit height are never IFS-attractors, as will be shown in the next section.

\section{Scattered spaces of limit height}

\begin{theorem}\label{TwLimit}
A compact scattered metric space of limit Cantor-Bendix\-son height is not homeomorphic to any IFS-attractor consisting of weak contractions.
In particular, it is not a topological IFS-attractor.
\end{theorem} 

\begin{proof}
Due to Mazurkiewicz-Sierpi\'nski's Theorem, such a space is of the form $\K=\omega^\delta\cdot n+1$, where $\delta=\h(\K)$ is a limit ordinal. We assume that $\K$ has a fixed metric $d$ and $\F$ is an IFS on $\K$ consisting of weak contractions. Suppose that  $\K=\bigcup_{f\in\F}f(\K)$, so there exists a weak contraction $f\in\F$ such that $\h(f(\K))=\delta$.  
Consequently the set $\F_1=\{f\in\F;\h(f(\K))=\delta\}$ is nonempty. Let $\F_0=\F\setminus\F_1$ and $\mu=\max(\{0\}\cup\{\h(f(\K)):~f\in\F_0\})$.
Then $\mu < \delta$, because $\delta$ is a limit ordinal. We will consider the Cantor-Bendixson rank $\rk(x)$ of points with respect to the space $\K$. Denote by $D$ the set of points of rank $\delta$. Note that the set $D=\K^{(\delta)}$ is finite.

\begin{claim}\label{CLemLimit}
If $f\in\F_1$ then $D\cap f(D)\neq\emptyset$.
\end{claim}

\begin{proof}
Suppose $D\cap f(D)=\emptyset$, so for every $x\in D$ the rank of $f(x)$ is less than $\delta$. Choose a neighborhood $V_x$ of $f(x)$ such that $V_x\cap D = \emptyset$. Then $\h(V_x)<\delta$. Find a clopen neighborhood $W_x$ of $x$ such that $f(W_x)\subset V_x$. Then $\h(f(W_x))$ is also less than $\delta$. Define $W=\bigcup_{x\in D}W_x$. Then the set $f(\K)=f(W)\cup f(\K\setminus W)$ has height $<\delta$, which gives a contradiction.
\end{proof}

We now come back to the proof of Theorem \ref{TwLimit}.

If the set $D=\{x_0\}$ is a singleton, then  by Claim \ref{CLemLimit} we know that $f(x_0)=x_0$ for every $f\in\F_1$. Then let $\rho$ be such that $\delta>\rho>\mu$. In the case where the set $D$ consists of more than one element, there exists 
$$\e=\min\{d(x,y)\colon~x\neq y,~ x,y\in D\cup\bigcup_{f\in\F_1}f(D)\}>0.$$ Due to the fact that $D=\K^{(\delta)}=\bigcap_{\rho<\delta}\K^{(\rho)}$ there exists an ordinal $\rho$ such that $\mu<\rho<\delta$ and $\K^{(\rho)}\subset\bigcup_{x\in D}B(x,\frac\e{2})$. Denote this set by $A$ so $$A=\K^{(\rho)}=\{x\in\K\colon\rk(x)\geq\rho\}\subset\bigcup_{x\in D}B(x,\frac\e{2}).$$ 
It is clear that $A$ is closed in $\K$ and $A\setminus D$ is nonempty set because $\delta$ is a limit ordinal number. Define $$\alpha=\sup_{x\in A}d(x,D)>0.$$ The set $A$ is compact, so there exists an element $a\in A$ such that $d(a,D)=~\alpha$ and $\rho\leq~\rk(a)<\delta$. It means that for every open neighborhood $U$ of $a$ we have $\h(U)\geq \rk(a)\geq\rho$.

Note that for $f\in\F_1$ if $a\in f(\K)$, then distance between set $f^{-1}(a)$ and set $D$ is greater than $\alpha$. 
Indeed for each $x\in f^{-1}(a)$ there exist $x_0,a_0\in D$ such that $d(x,x_0)=d(x,D)$ and $d(a,a_0)=d(a,D)=\alpha$. We first do the case $f(x_0)\in D$. Then we have
$$d(x,x_0)>d(f(x),f(x_0))= d(a,f(x_0))\geq d(a,D) =\alpha.$$ 
In the case $f(x_0)\notin D$ the set $D$ has more than one element. Note that $\alpha\leq~\frac\e{2}$, because $A\subset\bigcup_{x\in D}B(x,\frac\e{2})$. Moreover $d(a_0,f(x_0))\geq\e$ by the definition of $\e$. Then by the weak contracting property of $f$ and by the triangle inequality we have
$$d(x,x_0)>d(a,f(x_0))\geq d(a_0,f(x_0))-d(a,a_0)\geq\e-\alpha\geq\frac\e{2}\geq\alpha.$$
Consequently $d(x,D)>\alpha$ for every $x\in f^{-1}(a)$.

Thanks of that we can find a clopen neighborhood $U$ of $a$, such that $f^{-1}(U)\cap A$ is empty for every $f\in\F_1$. It implies that $\h(f^{-1}(U))<\rho$.  
\newline

The space $\K$ is an attractor of IFS $\F$, so we have $U=\bigcup_{f\in\F}f(f^{-1}(U))$. For $f\in\F_0$ it holds $\h(f(f^{-1}(U)))\leq \h(f(\K))\leq\mu<\rho$. For $f\in\F_1$ we know that $\h(f(f^{-1}(U)))\leq \h(f^{-1}(U))<\rho$. We finally have a contradiction by applying the fact that
$$\h(U)=\max_{f\in\F}\{\h(f(f^{-1}(U)))\}<\rho.$$
This completes the proof.
\end{proof}

\section{Scattered spaces of successor height}

Recall that every countable scattered compact space is homeomorphic to an ordinal $\omega^{\beta}\cdot n + 1$, with the order topology.
We start with the case $n=1$.

\begin{theorem}\label{Tw2}
For every $\e>0$ and every countable ordinal $\delta$ the scattered space $\omega^{\delta+1}+1$ is homeomorphic to the~attractor of an iterated function system consisting of two contractions $\{\ph,\ph_{\delta+1}\}$ in the unit interval $I=[0,1]$, such that $$\max(\Lip(\ph),\Lip(\ph_{\delta+1}))<\e.$$
\end{theorem}

To prove this theorem we shall use the notion of a monotone ladder system.
We shall denote by $\LIM(\alpha)$ the set of all limit ordinals $\leq \alpha$.

\begin{definition}
Let $\alpha$ be an ordinal.
A \emph{monotone ladder system} in $\alpha$ is a collection of sequences $\{c_n^\alpha(\beta): n\in\N;~\beta \in \LIM(\alpha)\}$ such that 
\begin{itemize}
\item for each ordinal $\beta \in \LIM(\alpha)$ the sequence $\{c_n^\alpha(\beta)\}_{n\in\N}$ is strictly increasing and converges to $\beta$ when $n\rightarrow\infty$;
\item for every $\beta,\gamma\in \LIM(\alpha)$ if $\beta\leq\gamma$ then $c_n^\alpha(\beta)\leq c_n^\alpha(\gamma)$ for every $n\in\N$.
\end{itemize}
\end{definition}

%Let us note that it is not possible to have a monotone ladder system in the set of all countable ordinals.

We shall need monotone ladder systems for our construction. Their existence is rather standard, we give a proof for the sake of completeness.

\begin{lemma}\label{LemmaMLS}
For every countable ordinal $\alpha$ there exists a monotone ladder system in $\alpha$.
\end{lemma}

\begin{proof}
We prove that lemma by transfinite induction on limit ordinals $\leq \alpha$.
Setting $c_n^\omega(\omega) = n$, we obtain a monotone ladder system in $\omega$.

Now suppose that $\alpha$ is a limit ordinal and for all limit ordinals $\alpha'<\alpha$ there exists a monotone ladder system $\{c_n^{\alpha'}(\beta): n\in\N;~\beta \in \LIM(\alpha')\}$ in~$\alpha'$. We have to construct such a system in $\alpha$. 

If $\alpha=\alpha'+\omega$ then
$$c_n^\alpha(\beta)=c_n^{\alpha'}(\beta) \text{ for every } \beta \in \LIM(\alpha')$$
and
$$c_n^\alpha(\alpha)=\alpha'+n.$$
It is obvious that this is a monotone ladder system in $\alpha$.

Now suppose that $\alpha$ is a limit ordinal among limit ordinals and choose a strictly increasing sequence $\{\alpha_n\}_{n\in\N}$ such that $\alpha_0 = 0$, $\alpha_n \in \LIM(\alpha)$ for $n > 0$ and $\alpha = \sup_{n \in \N}\alpha_n$.

Given a limit ordinal $\beta<\alpha$ there exists a natural number $n_0$ such that $\alpha_{n_0}<\beta\leq\alpha_{n_0+1}$. Let
$$\bar{c}_n(\beta)=\max(\alpha_{n_0},c_n^{\alpha_{n_0+1}}(\beta)).$$
Note that $\{\bar{c}_n(\beta):n\in\N;~ \beta \in \LIM(\alpha),\; \beta < \alpha\}$ is a monotone ladder system: for any limit ordinals $\beta\leq\gamma<\alpha$ there exist $n_0,m_0\in\N$ such that $\alpha_{n_0}<\beta\leq\alpha_{n_0+1}$ and $\alpha_{m_0}<\gamma\leq\alpha_{m_0+1}$. If $n_0<m_0$ then for all $n\in\N$
$$\bar{c}_n(\beta)\leq\alpha_{n_0+1}\leq\alpha_{m_0}\leq\bar{c}_n(\gamma).$$
If $n_0=m_0$ then by the inductive hypothesis for $\alpha_{n_0+1}$ we have
$$\bar{c}_n(\beta)=\max(\alpha_{n_0},c_n^{\alpha_{n_0+1}}(\beta))\leq\max(\alpha_{m_0},c_n^{\alpha_{m_0+1}}(\gamma))=\bar{c}_n(\gamma).$$
Now we construct a monotone ladder system in $\alpha$ as follows. For every $\beta<\alpha$ and $n\in\N$ define
$$c_n^\alpha(\beta)=\min(\alpha_n,\bar{c}_n(\beta)) \text{ and } c_n^\alpha(\alpha)=\alpha_n.$$ 
Note that for every limit ordinals $\beta\leq\gamma\leq\alpha$, if $\gamma<\alpha$ then $c_n^\alpha(\beta)\leq c_n^\alpha(\gamma)$ because  $\{\bar{c}_n(\beta):n\in\N;~ \beta \in \LIM(\alpha),\; \beta<\alpha\}$ was monotone. If $\gamma=\alpha$ then $c_n^\alpha(\beta)=\min(\alpha_n,\bar{c}_n(\beta))\leq \alpha_n = c_n^\alpha(\alpha)$.
This means that $\{c_n^{\alpha}(\beta): n\in\N;~\beta \in \LIM(\alpha)\}$ is indeed a monotone ladder system in $\alpha$.

Note that for any limit ordinal $\alpha$, its monotone ladder system is also a monotone ladder system in every successor ordinal $\beta$, such that $\alpha<\beta<\alpha+\omega$. 
\end{proof}

\begin{proof}[Proof of Theorem \ref{Tw2}]
Fix a countable ordinal $\delta$. We have to construct an IFS-attractor homeomorphic to the space $\omega^{\delta+1}+1$. By Lemma~\ref{LemmaMLS} there exists a monotone ladder system for $\delta'=\delta+\omega$. Then for every ordinals $\alpha\leq\delta$ define a sequence $\{\alpha_n\}_{n\in\N}$ such that 
\begin{itemize}
\item if $\alpha$ is a limit ordinal, we put $\alpha_n:=c_n^{\delta'}(\alpha)$
\item if $\alpha$ is a successor ordinal, we put $\alpha_n:=c_n^{\delta'}(\alpha+\omega)$
\end{itemize}
Note that for every $\alpha,\beta\leq\delta$ if $\alpha\leq\beta$ then $\alpha_n\leq\beta_n$ for all $n\in\N$.
\newline

Let $r>3$. For a natural number $n$ consider the affine homeomorphism
$$s_n(x)=\frac{x}{r^n}+\frac{1}{r^n}.$$
Now for every ordinal $\alpha\leq\delta+1$ we construct scattered compact sets $L_\alpha, K_\alpha\subset I$, homeomorphic to $\omega^\alpha+1$, as follows:
\begin{enumerate}[1.]
\item $L_0= \{0\}$ 
\item $L_{\alpha}=L_0\cup\bigcup_{\alpha_n<\alpha}s_n(L_{\alpha_n}) \cup\bigcup_{\alpha_n\geq\alpha}s_n(L_{\alpha'})$ for $\alpha=\alpha'+1$ successor ordinal
\item $L_\alpha = L_0\cup\bigcup_{n=1}^\infty s_n(L_{\alpha_n})$ for a limit ordinal $\alpha$.
\end{enumerate}
Now define
\begin{enumerate}[(a)]
\item $K_0=L_0$ 
\item $K_{\alpha+1}= K_0\cup\bigcup_{n=1}^\infty s_n(K_\alpha)$
\item $K_\alpha = L_\alpha$ for a limit ordinal $\alpha$.
\end{enumerate}
Each of these spaces consists of blocks contained in $s_n(I)$, each block is a space of a lower height and they accumulate to $0$. 

\begin{figure}[h]
	\includegraphics[scale =0.7]{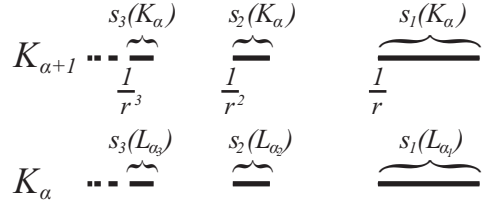}
\caption{The spaces $K_\alpha$}
\end{figure}

\begin{figure}[h]
	\includegraphics[scale =0.7]{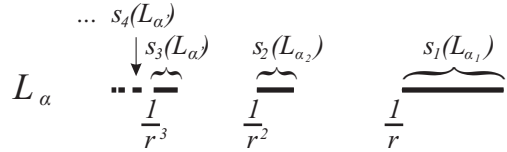}
	\caption{An example of $L_\alpha$ where $\alpha=\alpha'+1$ and $\alpha_2<\alpha\leq\alpha_3$}
\end{figure}

Now we make the following definition of an iterated function system $\{\ph,\ph_{\delta+1}\}$ such that $\ph(K_{\delta+1})\cup\ph_{\delta+1}(K_{\delta+1})=K_{\delta+1}$.
We use the contraction $$\ph(x)=\frac{x}r$$ that shifts every block contained in $s_n(I)$ onto the next block, contained in $s_{n+1}(I)$. In particular
$$\ph(K_{\delta+1})=  K_{\delta+1}\setminus s_1(K_{{\delta}}).$$
Now we define $\ph_{\delta+1}=s_1\circ f_\delta$ where $f_\delta$ is defined below, with the use of additional functions $g_\alpha$. Namely, for every $\alpha\leq\delta$ we define
\begin{enumerate}
\item $g_0 = \begin{cases}
0, & x\in[0,\frac{2}r]  \\
\frac{r}{r-2}(x-\frac2{r}), & x\in(\frac{2}r,1]
\end{cases}$ 
\item $g_{\alpha}(x)=
\begin{cases}
s_n(g_{\alpha_n}(s_n^{-1}(x))), & x\in s_n(I) \text{ and } \alpha_n<\alpha,\; n\geq 1 \\
s_n(g_{\alpha'}(s_n^{-1}(x))), & x\in s_n(I) \text{ and } \alpha_n\geq\alpha,\; n\geq 1 \\
x, & \text{otherwise}
\end{cases}$ 
\newline whenever $\alpha=\alpha'+1$ is a successor ordinal
\item $g_\alpha=f_\alpha$ for $\alpha$ a limit ordinal.
\end{enumerate}
Finally, define
\begin{enumerate}[1.]
\item $f_0=g_0$

\item  $f_{\alpha+1}(x)=
\begin{cases}
s_n(f_\alpha(s_n^{-1}(x))), & x\in s_n(I), \text{ for some }n\geq 1 \\
x, & \text{otherwise}
\end{cases}$

\item  $f_{\alpha}(x)=
\begin{cases}
s_n(g_{\alpha_n}(s_n^{-1}(x))), & x\in s_n(I), \text{ for some }n\geq 1 \\
x, & \text{otherwise}
\end{cases}$\\
for a limit ordinal $\alpha$.
\end{enumerate} 

\begin{figure}[h]
	\includegraphics[scale =0.45]{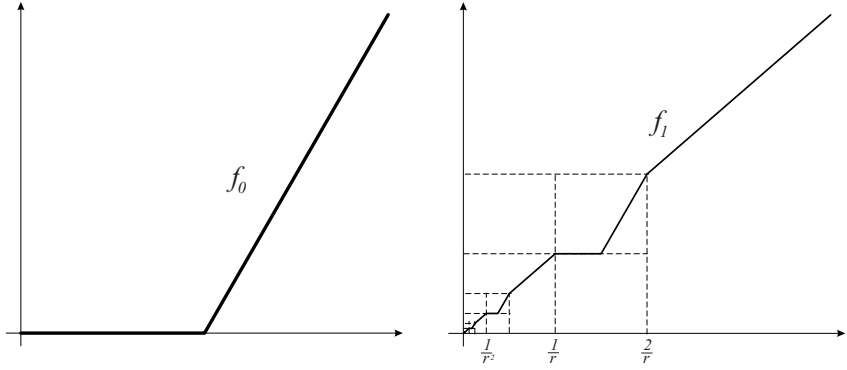}
	\caption{The functions $f_0$ and $f_1$ for $r=4$}
\end{figure}

Note that every function $f_\alpha$ and $g_\alpha$ is continuous and $\Lip(f_\alpha)=\Lip(g_\alpha)=\frac{r}{r-2}$, so $$\Lip(\ph_{\delta+1})=\Lip(s_1) \cdot \Lip(f_{\delta})=\frac{1}{r-2}<1.$$
Moreover $\max(\Lip(\ph),\Lip(\ph_{\delta+1}))=\frac{1}{r-2}$ thus for every $\e > 0$ we can find $r>3$, such that $\frac{1}{r-2}<\e$.

Now we show that for every ordinals $\alpha,\beta\leq\delta$ the following properties hold
\begin{enumerate}[(A)]
\item \label{K1} $g_\alpha(L_\beta)=L_\alpha$ when $\alpha\leq\beta$
\item \label{K3} $f_\alpha(K_{\alpha+1})=K_\alpha$. 
\end{enumerate} 

\begin{proof}[Proof of property (\ref{K1})]
The proof is by transfinite induction on $\beta$. For $\beta=0$ it is true that $g_0(L_0)=L_0$. 

In the second step we assume that for every $\beta'<\beta$  and each $\alpha'\leq\beta'$ it holds that $g_{\alpha'}(L_{\beta'})=L_{\alpha'}$. Let us consider four cases where $\alpha\leq\beta$. Note that in each case $\alpha_n\leq\beta_n$ for all $n\in\N$.
\newline
\textbf{Case 1.}
$\alpha$ and $\beta$ are limit ordinals (in particular $\alpha_n\nearrow\alpha$ and $\beta_n\nearrow\beta$).
Then by the inductive hypothesis
$$
g_\alpha(L_\beta) =
L_0\cup\bigcup_{n=1}^\infty s_n(g_{\alpha_n}(L_{\beta_n}))=
L_0\cup\bigcup_{n=1}^\infty s_n(L_{\alpha_n})=L_\alpha.
$$
\newline
\textbf{Case 2.}
$\alpha=\alpha'+1$ and $\beta$ is a limit ordinal. Then $\beta_n\nearrow\beta$ and again using the inductive hypothesis, we get
\begin{align*}
g_\alpha(L_\beta) 
&= L_0\cup\bigcup_{\alpha_n<\alpha}s_n(g_{\alpha_n}(L_{\beta_n})) \cup\bigcup_{\alpha_n\geq\alpha}s_n(g_{\alpha'}(L_{\beta_n}))=\\
&= L_0\cup\bigcup_{\alpha_n<\alpha}s_n(L_{\alpha_n}) \cup\bigcup_{\alpha_n\geq\alpha}s_n(L_{\alpha'})=L_\alpha.
\end{align*}
\newline
\textbf{Case 3.}
$\alpha$ is a limit ordinal and $\beta=\beta'+1$. Then $\alpha_n\nearrow\alpha$ and every $\alpha_n<\beta'$. Thus
\begin{align*}
g_\alpha(L_\beta) 
&= L_0\cup\bigcup_{\beta_n<\beta}s_n(g_{\alpha_n}(L_{\beta_n})) \cup\bigcup_{\beta_n\geq\beta}s_n(g_{\alpha_n}(L_{\beta'}))=\\
&= L_0\cup\bigcup_{\beta_n<\beta}s_n(L_{\alpha_n}) \cup\bigcup_{\beta_n\geq\beta}s_n(L_{\alpha_n})=\\
&=L_0\cup\bigcup_{n=1}^\infty s_n(L_{\alpha_n})=L_\alpha.
\end{align*}
\newline
\textbf{Case 4.}
$\alpha=\alpha'+1$ and $\beta=\beta'+1$. Then
\begin{align*}
g_\alpha(L_\beta)
&= g_\alpha(L_0\cup\bigcup_{\beta_n<\beta}s_n(L_{\beta_n}) \cup\bigcup_{\beta_n\geq\beta}s_n(L_{\beta'})) =\\
&=L_0\cup\bigcup_{\alpha_n<\alpha,\beta_n<\beta}s_n(g_{\alpha_n}(L_{\beta_n})) 
\cup\bigcup_{\alpha\leq\alpha_n,\beta_n<\beta}s_n(g_{\alpha'}(L_{\beta_n})) \cup\\
&\cup\bigcup_{\alpha_n<\alpha,\beta\leq\beta_n}s_n(g_{\alpha_n}(L_{\beta'})) \cup\bigcup_{\alpha\leq\alpha_n,\beta\leq\beta_n}s_n(g_{\alpha'}(L_{\beta'})).
\end{align*}
For each of the unions above we can use the inductive hypothesis and we get
$$g_\alpha(L_\beta) 
=L_0\cup\bigcup_{\alpha_n<\alpha}s_n(L_{\alpha_n}) \cup\bigcup_{\alpha_n\geq\alpha}s_n(L_{\alpha'}) = L_\alpha,$$
which completes the proof of property (\ref{K1}).
\end{proof}

\begin{proof}[Proof of property (\ref{K3})] Once again we use transfinite induction. For $\alpha=0$ it is obvious that $f_0(K_1)=K_0$, because $K_1\subset[0,\frac{2}r]=f_0^{-1}(K_0)$. Therefore, if $\alpha=\alpha'+1$, then by the inductive hypothesis
$$f_\alpha(K_{\alpha+1})=K_0\cup\bigcup_{n=1}^\infty s_n(f_{\alpha'}(K_{\alpha'+1})) = K_0\cup\bigcup_{n=1}^\infty s_n(K_{\alpha'})=K_\alpha.$$
If $\alpha$ is a limit ordinal then, using property (\ref{K1}), we get
\begin{align*}
f_\alpha(K_{\alpha+1}) &= K_0\cup\bigcup_{n=1}^\infty s_n(g_{\alpha_n}(K_{\alpha})) = K_0\cup\bigcup_{n=1}^\infty s_n(g_{\alpha_n}(L_{\alpha})) =\\
&= K_0\cup\bigcup_{n=1}^\infty s_n(L_{\alpha_n})=K_\alpha,
\end{align*}
which completes the proof of property (\ref{K3}).
\end{proof}

Finally, we show that the scattered space $K_{\delta+1}$ is the~attractor of $\{\ph,\ph_{\delta+1}\}$.
Indeed, using property (\ref{K3}) we obtain that
\begin{align*} 
\ph(K_{\delta+1})\cup\ph_{\delta+1}(K_{\delta+1}) &=\big(K_{\delta+1}\setminus s_1(K_\delta)\big)\cup s_1(f_\delta(K_{\delta+1}))=\\
&=\big(K_{\delta+1}\setminus s_1(K_{\delta})\big)\cup s_1(K_{\delta})=K_{\delta+1}.
\end{align*}
This finishes the proof of Theorem \ref{Tw2}.
\end{proof}

The space $\omega^\alpha\cdot n +1$ can be represented as the union of $n$ disjoint copies of $\omega^\alpha +1$.
In view of Lemma~\ref{LemYUnion}, such a space is an IFS-attractor whenever it is properly embedded into the real line (or some other metric space).

Summarizing:

\begin{corollary}
A countable compact space $X$ is a topological IFS-attractor if and only if its Cantor-Bendixson height is a successor ordinal.
If this is the case, then $X$ is homeomorphic to an IFS-attractor in the real line.
\end{corollary}

As we have already mentioned, taking the space $\omega^{\omega+1} + 1$, we obtain an example of a countable IFS-attractor with a clopen set (homeomorphic to $\omega^\omega+1$) that is not a topological IFS-attractor.

\subsection*{Acknowledgements}

The author would like to thank Wies{\l}aw Kubi\'s and Taras Banakh for several fruitful discussions.

\end{document}